\documentclass{amsart}
\usepackage[latin1]{inputenc}

\usepackage{amsfonts,enumerate}

\usepackage[mathscr]{euscript}
\usepackage{epsfig}
\usepackage{latexsym}

\usepackage[all]{xy}

\usepackage{amssymb}
\usepackage{amsthm}
\usepackage{amsmath}
\usepackage{amsxtra}

.tfm
.tfm

\theoremstyle{plain}
\newtheorem{prop}{Proposition}[section]
\newtheorem{thm}[prop]{Theorem}

\newtheorem{lemma}[prop]{Lemma}

\theoremstyle{definition}
\newtheorem*{defi}{Definition}
\newtheorem*{notation}{Notation}

\theoremstyle{remark}

\newtheorem{remark}{Remark}

\numberwithin{table}{section}

\DeclareMathOperator{\Frob}{Frob}
\DeclareMathOperator{\Gal}{Gal}
\DeclareMathOperator{\im}{Im}

\DeclareMathOperator{\identity}{id}

\newcommand{\bfP}{\mathcal P}

\newcommand{\E}{\mathcal{E}}

\newcommand{\Om}{{\mathscr{O}}}

\newcommand{\Disc}{\Delta}

\newcommand{\GL}{{\rm GL}}
\newcommand{\GSp}{{\rm GSp}}

\def\ZZ{\mathbb Z}

\def\FF{\mathbb F}
\def\QQ{\mathbb Q}
\def\AA{\mathbb A}
\def\CC{\mathbb C}

\def\<#1>{{\left\langle{#1}\right\rangle}}

\def\Z{{\mathbb Z}}             
\def\Q{{\mathbb Q}}             

\def\set#1{{\left\{{\def\st{\;:\;}#1}\right\}}}

\def\O{{R}}           
\def\Ol(#1){{\mathop{\O_l}(\id{#1})}}                 
\def\Or(#1){{\mathop{\O_r}(\id{#1})}}                 
\def\Orp(#1){{\mathop{\O_r}(\idp{#1})}}               
\def\Otern(#1){{\mathop{\O^0_r}(\id{a})}}    

\def\id#1{{\mathfrak{#1}}}      

\DeclareMathOperator{\trace}{{\mathrm{Tr}}}

\begin{document}

\title[Modularity over imaginary quadratic fields]{Proving modularity for a given elliptic
  curve over an imaginary quadratic field}


\author{Luis Dieulefait}
\address{Departament d'Àlgebra i Geometria, Facultat de Matemàtiques,
  Universitat de \break Barcelona, Gran Via de les Corts Catalanes,
  585. 08007 Barcelona} 
\email{ldieulefait@ub.edu}
\author{Lucio Guerberoff}
\address{Departamento de Matem\'atica, Universidad de Buenos Aires,
  Pabell\'on I, Ciudad Universitaria. C.P:1428, Buenos Aires,
  Argentina - Institut de Mathématiques de Jussieu, Université Paris
  7, Denis Diderot, 2, place Jussieu, F-75251 PARIS CEDEX 05 France.}
\email{lguerb@dm.uba.ar}
\thanks{The second author was supported by a CONICET fellowship}
\author{Ariel Pacetti}
\address{Departamento de Matem\'atica, Universidad de Buenos Aires,
         Pabell\'on I, Ciudad Universitaria. C.P:1428, Buenos Aires, Argentina}
\email{apacetti@dm.uba.ar}
\thanks{The third author was partially suported by PICT 2006-00312 and
UBACyT X867}
\keywords{Elliptic Curves Modularity}
\subjclass[2000]{Primary: 11G05; Secondary:11F80}

\begin{abstract}
We present an algorithm to determine if the $L$-series associated to
an automorphic representation and the one associated to an elliptic
curve over an imaginary quadratic field agree. By the work of
Harris-Soudry-Taylor, Taylor and Berger-Harcos
(cf. \cite{harris-taylor}, \cite{taylorII} and \cite{berger-harcos})
we can associate to an automorphic representation a family of
compatible $p$-adic representations. Our algorithm is based on
Faltings-Serre's method to prove that $p$-adic Galois representations
are isomorphic.
\end{abstract}

\maketitle

\section{Introduction}

Modularity for rational elliptic curves was one of the biggest
achievements of last century. Little is known for general number
fields. In the case of totally real number fields some techniques do
apply, but the result is far from being proven. The case of not
totally real fields is more intractable to Taylor-Wiles machinery. In
this paper we present an algorithm to determine if the $L$-series
associated to an automorphic representation and the one associated to
an elliptic curve over an imaginary quadratic field agree or not. The
algorithm is based on Faltings-Serre's method to prove isomorphism of
$p$-adic Galois representation. By the work of Harris-Soudry-Taylor,
Taylor and Berger-Harcos (cf. \cite{harris-taylor}, \cite{taylorII}
and \cite{berger-harcos}) we can associate to an automorphic
representation a family of compatible $p$-adic representations, and an
elliptic curve has such a family of representations as well in the
natural way.

The paper is organized as follows: on the first section we present the
algorithms (which depend on the residual representations). On the
second section we review the results of $p$-adic representations
attached to automorphic forms on imaginary quadratic fields. On the
third section we explain Falting-Serre's method on Galois
representations. On the fourth section we prove that the algorithm
gives the right answer. At last we show some examples and some GP code
writen for the examples.

\section{Algorithm}

Let $K$ be an imaginary quadratic field, $\E$ be an elliptic curve over
$K$ and $f$ an automorphic form on $\GL_2(\AA_K)$ whose $L$-series we
want to compare. This algorithm answers if the $2$-adic Galois
representations attached to both objects are isomorphic and if the
original $L$-series are equal. Since these Galois representations come
in compatible families, in particular the algorithm determines whether
the $p$-adic Galois representations are isomorphic or not for any
prime $p$. It depends on the residual image of the elliptic curve
representation.

The input in all cases is: $K$, $\E$, $\id{n}(\E)$ (the conductor of
$\E$), $\id{n}(f)$ (the level of $f$) and $a_{\id{p}}(f)$ for some prime
ideals $\id{p}$ to be determined. By $L_\E$ we denote the field
obtained from $K$ by adding the coordinates of the $2$-torsion points
of $\E$.

\begin{notation} By $\bar \QQ$ we denote an algebraic closure of
  $\QQ$. Let $K$ be an imaginary quadratic extension of $\QQ$, and
  $\alpha$ an element of $K$. By $\bar \alpha$ we denote conjugate of
  $\alpha$. 

Let $L/K$ be field extensions and $\id{p} \subset \Om_K$. For $\id{q}
  \subset \Om_L$ a prime ideal above $\id{p}$, we denote
  $e(\id{q}|\id{p})$ the ramification index.
\end{notation}

\subsection{Residual image isomorphic to $S_3$}
\label{s3-case}
\begin{enumerate}
\item Let $\id{m}_K \subset \Om_K$ be given by $\id{m}_K:=
  \prod_{\id{p} \mid 2 \id{n}(\E) \id{n}(f) \overline{\id{n}(f)}\Disc(K)}
\id{p}^{e(\id{p})}$ 
  where 
\[
e(\id{p}) = \left\{ \begin{array}{cl}
1 & \text{if }\id{p} \nmid 6\\
2e(\id{p}|2)+1 & \text{if }\id{p} \mid 2\\
\left \lfloor \frac{3 e(\id{p}|3)}{2} \right \rfloor +1 & \text{if }\id{p} \mid 3. \end{array} \right.
\]
Compute the ray class group $Cl(\Om_K,\id{m}_K)$.
\item Identify the character $\psi$ corresponding to the unique quadratic
  extension of $K$ contained on $L_\E$ on the computed basis.
\item Extend $\set{\psi}$ to a basis $\set{\psi,\chi_i}_{i=1}^n$ of
the quadratic characters of $Cl(\Om_K,\id{m})$. Compute prime ideals
$\set{\id{p}_j}_{j=1}^{n'}$ with $\id{p}_j \subset \Om_K$, $\id{p}_j
\nmid \id{m}_K$, and with inertial degree $3$ on $L_\E$ such that
\[
\<(\log(\chi_1(\id{p}_j)),\ldots,\log(\chi_n(\id{p}_j)))>_{j=1}^{n'} =(\ZZ/2\ZZ)^n
\]
(where we take any root of the logarithm and identify $\log(\pm1)$ with
$\ZZ/2\ZZ$).
\item If $\trace(\rho_f(\Frob_{\id{p}_j}))$ is odd for $1 \le j \le
  n'$, $\tilde \rho_f$ has image isomorphic to $C_3$ or to $S_3$ with
  the same intermediate quadratic field as $\tilde \rho_\E$. If not,
  end with output ``the two representations are not isomorphic''.
\item Compute a basis $\{\chi_i\}_{i=1}^m$ of cubic characters of
  $Cl(\Om_K,\id{m}_K)$ and a set of ideals $\set{\id{p}_i}_{i=1}^{m'}$ such
  that $\psi(\id{p}_i)=-1$ or $\id{p}_i$ splits completely on $L_\E$ and
\[
\<(\log(\chi_1(\id{p}_j)),\ldots,\log(\chi_m(\id{p}_j)))>_{j=1}^{m'} =(\ZZ/3\ZZ)^m.
\]
\item If $\trace(\rho_f(\Frob_{\id{p}_j}))$ is even for $1 \le j \le m'$, $\tilde
  \rho_f$ has $S_3$ image with the same intermediate quadratic field
  as $\tilde \rho_\E$. If not, end with output ``the two representations are
  not isomorphic''.
\item Let $K_\E$ be the degree two extension of $K$ contained in $L_\E$
  and $\id{m}_{K_\E} \subset \Om_{K_\E}$ be given by $\id{m}_{K_\E}:=
  \prod_{\id{p} \mid 2 \id{n}(\E) \id{n}(f) \overline{\id{n}(f)}\Disc(K)}
\id{p}^{e(\id{p})}$ where
\[
e(\id{p}) = \left\{ \begin{array}{cl}
1 & \text{if }\id{p} \nmid 3\\
\left \lfloor \frac{3 e(\id{p}|3)}{2} \right \rfloor +1 & \text{if }\id{p} \mid 3. \end{array} \right.
\]
Compute the ray class group $Cl(\Om_{K_\E},\id{m}_{K_\E})$.
\item Identify the character $\psi_\E$ corresponding to the cubic
extension $L_\E$ on the computed basis and extend it to a basis
$\set{\psi_\E,\chi_i}_{i=1}^m$ of order three characters of
$Cl(\Om_{K_\E},\id{m}_{K_\E})$. Compute prime ideals
$\set{\id{p}_j}_{j=1}^{m'}$ with $\id{p}_j \subset \Om_K$,
$\psi_\E(\id{p}_j)=1$ and such that
\[
\<(\log(\chi_1(\id{p}_j)),\ldots,\log(\chi_n(\id{p}_j)))>_{j=1}^{m'} =(\ZZ/3\ZZ)^m
\]
(where we take any identification of the cubic roots of unity with
$\ZZ/3\ZZ$). If $\trace(\rho_f(\Frob_{\id{p}_j})) \equiv
\trace(\rho_\E(\Frob_{\id{p}_j})) \pmod 2$
for $1 \le j \le m'$, both residual representations are
isomorphic. If not, end with output ``the two representations are not
isomorphic''.
\item Let $\id{m}_{L_\E}\subset \Om_{L_\E}$ be the modulus
  $\id{m}_{L_\E}= \prod_{\id{q} \mid 2 \id{n}(\E) \id{n}(f)
  \overline{\id{n}(f)}\Disc(K)} \id{q}^{e(\id{q})}$ where
\[
e(\id{p}) = \left\{ \begin{array}{cl}
1 & \text{if }\id{p} \nmid 2\\
2e(\id{p}|2)+1 & \text{if }\id{p} \mid 2. \end{array} \right.
\]
Compute the ray class group $Cl(\Om_{L_\E},\id{m}_{L_\E})$. Let
$\set{\chi_i}_{i=1}^n$ be a basis for its quadratic characters (dual
to the ray class group one computed).
\item Compute the Galois group $\Gal(L_\E/K)$.
\item (Computing invariant subspaces) Let $\sigma$ be an order $3$
  element of $\Gal(L_\E/K)$ and solve the homogeneous system
\[
\left(
\begin{array}{ccc}
\log(\chi_1(\id{a}_1\sigma(\id{a}_1))) & \ldots & \log(\chi_n(\id{a}_1\sigma(\id{a}_1)))\\
\vdots & &\vdots\\
\log(\chi_1(\id{a}_n\sigma(\id{a}_n))) & \ldots & \log(\chi_n(\id{a}_n\sigma(\id{a}_n)))
\end{array} \right)
\]
Denote by $V_{\sigma}$ the kernel.
\item Take $\tau$ an order $2$ element of $\Gal(L/K)$ and compute
  $V_{\tau}$, the kernel of the same system for $\tau$.
\item Intersect $V_{\sigma}$ with $V_{\tau}$. Let $\set{\chi_i}_{i=1}^m$ be a basis of the
  intersection. This gives generators for the $S_3 \times C_2$
  extensions.
\item Compute a set of ideals $\set{\id{p}_i}_{i=1}^{m'}$ with
  $\id{p}_i \subset \Om_K$ and $\id{p}_i \nmid \id{m}_K$ such that 
\[
\<(\log(\chi_1(\tilde{\id{p}}_j)),\ldots,\log(\chi_n(\tilde{\id{p}}_j)))>_{j=1}^{m'} =(\ZZ/2\ZZ)^m,
\]
where $\tilde{\id{p}}_i$ is any ideal of $L_\E$ above $\id{p}_i$.
\item If $\trace(\rho_f(\Frob_{\id{p}_i})) =
  \trace(\rho_\E(\Frob_{\id{p}_i}))$ for $1 \le i \le m$ then the two
  representations agree on order $6$ elements, else end with output
  ``the two representations are not isomorphic''.
\item For $\sigma$ an order three element, solve the homogeneous system
\[
\left(
\begin{array}{ccc}
\log(\chi_1(\id{a}_1\sigma(\id{a}_1)\sigma^2(\id{a}_1))) & \ldots & \log(\chi_n(\id{a}_1\sigma(\id{a}_1)\sigma^2(\id{a}_1)))\\
\vdots & &\vdots\\
\log(\chi_1(\id{a}_n\sigma(\id{a}_n))\sigma^2(\id{a}_n)) & \ldots & \log(\chi_n(\id{a}_n\sigma(\id{a}_n)\sigma^2(\id{a}_n)))
\end{array} \right)
\]
Denote by $W_{\sigma}$ such kernel.
\item Intersect $W_{\sigma}$ with $V_{\tau}$. Let
$\set{\chi_i}_{i=1}^t$ be a basis of such subspace. This characters
give all the $S_4$ extensions.
\item Compute a set of ideals $\set{\id{p}_i}_{i=1}^{t'}$ with
  $\id{p}_i \subset \Om_K$ and $\id{p}_i \nmid \id{m}_K$ such that 
\[
\<(\log(\chi_1(\tilde{\id{p}}_j)),\ldots,\log(\chi_n(\tilde{\id{p}}_j))),\ldots,(\log(\chi_1(\sigma^2(\tilde{\id{p}}_j))),\ldots,\log(\chi_n(\sigma^2(\tilde{\id{p}}_j))))>_{j=1}^{t'}
\]
equals $(\ZZ/2\ZZ)^t$, where $\tilde{\id{p}}_i$ is any ideal of $L_\E$ above
$\id{p}_i$.
\item If $\trace(\rho_f(\id{p}_i)) = \trace(\rho_\E(\id{p}_i))$ for all $1
  \le i \le n$ output ``$\rho_f \simeq \rho_\E$''. If not output ``the
  two representations are not isomorphic''.
\item If $a_{\id{p}}(f)=a_{\id{p}}(\E)$ for $\id{p} \mid
  2\id{n}(\E)\id{n}(f) \overline{\id{n}(f)}\Delta(K)$ then $L(f,s) = L(\E,s)$.
\end{enumerate}

\subsection{Residual image trivial or isomorphic to $C_2$}
\label{trivial-case}
\begin{enumerate}[(1)]
\item Chose prime ideals $\bfP_i$, $i=1,2$ such that $2$ has no
  inertial degree on $\QQ[\alpha_i]$, where $\alpha_i$ is a root of
  $\Frob_{\bfP_i}$. If $\trace(\rho_\E(\Frob_{\bfP_i})) \neq
  \trace(\rho_f(\Frob_{\bfP_i}))$, end with output ``the two
  representations are not isomorphic''.
\item Let $\id{m}_K \subset \Om_K$ be given by $\id{m}_K:=
  \prod_{\id{p} \mid 2 \id{n}(\E) \id{n}(f) \overline{\id{n}(f)}\Disc(K)} \id{p}^{e(\id{p})}$
  where 
\[
e(\id{p}) = \left\{ \begin{array}{cl}
1 & \text{if }\id{p} \nmid 6\\
2e(\id{p}|2)+1 & \text{if }\id{p} \mid 2\\
\left \lfloor \frac{3 e(\id{p}|3)}{2} \right \rfloor +1 & \text{if }\id{p} \mid 3. \end{array} \right.
\]
Compute the ray class group $Cl(\Om_K,\id{m}_K)$.
\item For each index two subgroup of $Cl(\Om_K,\id{m}_K)$ (plus the whole
      group), take the corresponding quadratic (or trivial)
      extension $L$. In $L$, take the
      modulus $\id{m}_L=\prod_{\id{p} \mid 2 \Disc(K)\id{n}(\E) \id{n}(f) \overline{\id{n}(f)}} \id{p}^{e(\id{p})}$, where 
\[
e(\id{p}) = \left\{ \begin{array}{cl}
1 & \text{if }\id{p} \nmid 3\\
\left \lfloor \frac{3 e(\id{p}|3)}{2} \right \rfloor +1 & \text{if
}\id{p} \mid 3. \end{array} \right .
\]
  and compute the ray class group $Cl(\Om_L,\id{m}_L)$.
 \item Compute a set of generators $\{\chi_j\}_{j=1}^n$ for the cubic
      characters of $Cl(\Om_L,\id{m}_L)$, and find prime ideals
      $\{\id{q}_j\}_{j=1}^{n'}$ of $\Om_L$, with
      $\id{q}_j\nmid\id{m}_L$ and such that
      \[ \langle(\log(\chi_1(\id{q}_j)),\dots,\log(\chi_n(\id{q}_j)))\rangle_{j=1}^{n'}=(\Z/3\Z)^n. \]
 \item Consider the collection $\{\id{p}_1,\dots,\id{p}_m\}$ of all
      prime ideals of $\Om_K$ which are below the prime ideals
      found on step (3).
 \item If $\trace(\tilde{\rho}_f(\Frob_{\id{p}_i}))\equiv 0 \pmod 2$
      for each $i=1,\dots,m$, then $\tilde \rho_f$ has image trivial or
      isomorphic to $C_2$. Otherwise, output ``the two representations are not
      isomorphic''.
 \item Compute a basis $\{\chi_i\}_{i=1}^n$ of quadratic characters of
 $Cl(\Om_K,\id{m}_K)$. 
 \item Compute a set of prime ideals $\{\id{p}_{i} \subset \Om_K \, : \,\id{p}_{i}
   \nmid \id{m}_K\}_{i=1}^{2^n-1}$ such that  
\[
\{(\log(\chi_1(\id{p}_{i})), \ldots, \log(\chi_n(\id{p}_{i}))\}_{i=1}^{2^n-1} = (\ZZ/2\ZZ)^n\backslash\{0\}
\]
 \item If $\trace(\rho_f(\Frob_{\id{p}_i}))=\trace(\rho_\E(\Frob_{\id{p}_i}))$
      for $i=1,\dots,2^n-1$, $\rho_\E^{ss} \simeq \rho_f^{ss}$. If not, output ``the two
      representations are not isomorphic''.
\item If $a_{\id{p}}(f)=a_{\id{p}}(\E)$ for $\id{p} \mid
  2\id{n}(\E)\id{n}(f) \overline{\id{n}(f)}\Delta(K)$ then $L(f,s) = L(\E,s)$.
\end{enumerate}

\begin{remark} The algorithm can be slightly improved. In step $(8)$,
  instead of aiming at the whole $C_2^r$, we can stop when we reach a
  \emph{non-cubic} set.

  \begin{defi} Let $V$ be a finite dimensional vector space. A subset $T$
  of $V$ is called \emph{non-cubic} if each homogeneous polynomial 
    on $V$ of degree $3$ that is zero on $T$, is zero on $V$.
  \end{defi}

  In particular, the whole space $V$ is non-cubic. 

\subsection{Residual image isomorphic to $C_3$}
\label{c3-case}

\begin{enumerate}
\item Chose prime ideals $\bfP_i$, $i=1,2$ such that $2$ has no
  inertial degree on $\QQ[\alpha_i]$, where $\alpha_i$ is a root of
  $\Frob_{\bfP_i}$. If $\trace(\rho_\E(\Frob_{\bfP_i})) \neq
  \trace(\rho_f(\Frob_{\bfP_i}))$, end with output ``the two
  representations are not isomorphic''.
\item 
Let $\id{m}_K \subset \Om_K$ be given by 
$
\id{m}_K := \prod_{
  \id{p} \mid 2 \id{n}(\E) \id{n}(f) \overline{\id{n}(f)}\Disc(K)}
  \id{p}^{e(\id{p})} \text{, where }
$
\[
e(\id{p}) = \left\{ \begin{array}{cl}
1 & \text{if }\id{p} \nmid 6\\
2e(\id{p}|2)+1 & \text{if }\id{p} \mid 2\\
\left \lfloor \frac{3 e(\id{p}|3)}{2} \right \rfloor +1 & \text{if }\id{p} \mid 3. \end{array} \right.
\]

Compute the ray class group $Cl(\Om_K,\id{m}_K)$.
\item Identify the character $\psi_\E$ corresponding to the cubic
  Galois extension $L_\E$ on the computed basis.
\item Find a basis $\set{\chi_i}_{i=1}^n$ of the quadratic characters
of $Cl(\Om_K,\id{m}_K)$. Compute prime ideals $\set{\id{p}_j}_{j=1}^{n'}$
with $\id{p}_j \subset \Om_K$, $\id{p}_j \nmid \id{m}$, $\psi(\id{p}_j) \neq
1$ and such that
\[
\<(\log(\chi_1(\id{p}_j)),\ldots,\log(\chi_n(\id{p}_j)))>_{j=1}^{n'} =(\ZZ/2\ZZ)^n
\]
(where we take any root of the logarithm and identify $\log(\pm1)$ with
$\ZZ/2\ZZ$).
\item If $\trace(\rho_f(\Frob_{\id{p}_j}))$ is odd for $1 \le j \le n'$, $\tilde \rho_f$ has
image isomorphic to $C_3$. If not, end with output ``the two representations are
  not isomorphic''.
\item Extend $\{\psi_\E\}$ to a basis $\set{\psi_\E,\chi_i}_{i=1}^{m}$ of
order three characters of $Cl(\Om_K,\id{m}_K)$. Compute prime ideals
$\set{\id{p}_j}_{j=1}^{m'}$ with $\id{p}_j \subset \Om_K$,
$\psi_\E(\id{p}_j)=1$, such that
\[
\<(\log(\chi_1(\id{p}_j)),\ldots,\log(\chi_n(\id{p}_j)))>_{j=1}^{m'} =(\ZZ/3\ZZ)^m
\]
(where we take any root of the logarithm and identify log of the cubic
roots of unity with $\ZZ/3\ZZ$). If $\trace(\rho_f(\Frob_{\id{p}_j})) \equiv
\trace(\rho_\E(\Frob_{\id{p}_j})) \pmod{2}$ for $1 \le j \le m'$, the two residual
representations are isomorphic. If not, end with output ``the two
representations are not isomorphic''.

\item Apply the previous case, steps $(7)$ to $(10)$, with $K$ replaced
  by $L_\E$.

\end{enumerate}

\section{Sources of two-dimensional representations of $G_K$}

Let $K$ be an imaginary quadratic field. We want to consider
two-dimensional, irreducible, $p$-adic representations of the group
$G_K := \Gal(\bar{\Q}/K)$.\\ 
The first natural source of such representations comes from the action
of $G_K$ on the torsion points of an elliptic curve $\E$ defined over
$K$. More precisely, we consider the Tate module $T_p(\E)$ which is a
free rank two $\Z_p$-module with a $G_K$-action, thus gives rise to a
$p$-adic representation
$$ \rho_{\E,p} : G_K \rightarrow \GL_2(\Z_p). $$ In order to make sure
that the Galois representation $\rho_{\E,p}$ is absolutely irreducible
we will assume that $\E$ does not have Complex Multiplication. The
ramification locus of the representation $\rho_{\E,p}$ consists of
primes of $K$ dividing $p$ together with the set $S$ of primes of bad
reduction of $\E$. The family of Galois representations $\{ \rho_{\E,p}
\}$ is a compatible family and has conductor equal to the conductor of
the elliptic curve $\E$.\\ On the other hand, Harris-Soudry-Taylor,
Taylor and Berger-Harcos (cf. \cite{harris-taylor}, \cite{taylorII}
and \cite{berger-harcos}) have proved that one can attach compatible
families of two-dimensional Galois representations $\{ \rho_p \}$ to
any regular algebraic cuspidal automorphic representation $\pi$ of
$\GL_2(\AA_K)$, assuming that it has unitary central character
$\omega$ with $\omega = \omega^c$. As in the case of classical modular
forms ``to be attached" means that there is a correspondence between
the ramification loci of $\pi$ and the representation $\rho_p$ and
also that, at unramified places, the characteristic polynomial of
$\rho_p (\Frob_{\id{p}})$ agrees with the Hecke polynomial of $\pi$ at
$\id{p}$. However, since the method for constructing these Galois
representations depends on using a theta lift to link with automorphic
forms on $\GSp_4(\mathbb{A}_\Q)$, it can not be excluded that the
representations $\rho_p$ also ramify at the primes that ramify in
$K/\Q$. The precise statement of the result, valid only under the
assumption $\omega = \omega^c$, is the following (cf. \cite{taylorII},
\cite{harris-taylor} and \cite{berger-harcos}):

\begin{thm} 
\label{Harris-Taylor}
Let $S$ be the set of places in $K$ dividing $p$ or where $K/\Q$ or
$\pi$ or $\pi^c$ ramify. Then there exists an irreducible
representation:
$$ \rho_{\pi,p} : G_K \rightarrow \GL_2(\bar{\Q}_p) $$ such that if
$\id{p}$ is a prime of $K$ not in $S$ then $\rho_{\pi,p}$ is
unramified at $\id{p}$ and the characteristic polynomial of
$\rho_{\pi,p} (\Frob_{\id{p}})$ agrees with the Hecke polynomial
of $\pi$ at $\id{p}$
\end{thm}

\begin{remark}
Observe that, in particular, if for some prime $\id{p}$
ramifying in $K/\Q$ we happen to know that $\rho_{\pi,p}$ is
unramified at $\id{p}$, the above theorem does not imply that the
trace of $\rho_{\pi,p} (\Frob_{\id{p}})$ agrees with the Hecke
eigenvalue of $\pi$ at $\id{p}$, though it is expected that these
two values should agree. It is also expected that there is a conductor
for the family $\{ \rho_{\pi,p} \}$, i.e., that the conductor should
be independent of $p$ as in the case of elliptic curves. The value of
this conductor should also agree with the level of $\pi$.  
\end{remark}

\begin{remark}
Since the families of Galois representations attached to an elliptic
curve $\E$ over $K$ and to a cuspidal automorphic representation $\pi$
by the previous result are both compatible families, if one has for
one prime $p$ that $\rho_{\E,p} \cong \rho_{\pi,p}$ then the
same holds for every prime $p$.  
\end{remark}

\begin{remark}
Even if an automorphic representation $\pi$ as above has integer
eigenvalues and the right weight so that the attached Galois
representations ``look like" those attached to some elliptic curve,
one has to be careful because over imaginary fields such Galois
representations may correspond to a ``fake elliptic curve"
instead. Namely, such a two-dimensional Galois representation of $K$
may correspond to some abelian surface having Quaternionic
Multiplication over $K$, i.e., the action of $G_K$ on the $p$-adic
Tate module of $A$ is isomorphic to two copies of the Galois
representation.
\end{remark}
\begin{remark} At first the image is defined on a finite extension of
  $\Q_p$. Actually, it can be defined on the ring of integer of an at
  most degree $4$ extension $E_{\bfP}$ of $\Q_p$. Furthermore, let
  $v_i$, $i=1,2$ be two unramified paces of $K$ and let
  $\alpha_i,\beta_i$ be the roots of the characteristic polynomial of
  $\Frob_{v_i}$. If $\alpha_{v_i} \neq \beta_{v_i}$ and, in the case
  $v_i$ is split, $\alpha_{\overline{v_i}} + \beta_{\overline{v_i}}
  \neq 0$ then we can take $E=\Q[\alpha_{v_1},\alpha_{v_2}]$ and
  $E_{\bfP}$ as its completion at any prime above $p$ by Corollary $1$
  of \cite{taylorII}.
\label{Taylor}
\end{remark}

\section{Faltings-Serre's method}

\subsection{First case: the image is absolutely irreducible}

On this section we review Faltings-Serre's (\cite{serre}) method by stating the
main ideas of \cite{schutt} (Section 5) on our particular case. Let
\[
\rho_i: \Gal(\bar\Q/K) \rightarrow GL_2(\ZZ_l)
\]
be representations for $i=1,2$ such that they satisfy:
\begin{itemize}
\item They have the same determinant.
\item The mod $l$ reductions are absolutely irreducible and isomorphic.
\item There exists a prime $\id{p}$ such that
  $\trace(\rho_1(\Frob_{\id{p}})) \neq \trace(\rho_2(\Frob_{\id{p}}))$.
\end{itemize}
We want to give a finite set of candidates for $\id{p}$. Chose the
  maximal $r$ such that $\trace(\rho_1) \equiv \trace(\rho_2)
  \pmod{l^r}$, so we obtain a non-trivial map $\phi:\Gal(\overline
  \Q/K) \rightarrow \FF_l$ given by
\[
\phi(\sigma) \equiv \frac{\trace(\rho_1(\sigma)) -
  \trace(\rho_2(\sigma))}{l^r} \pmod l.
\]
If we assume that $\tilde{\rho_1} = \tilde{\rho_2}$, we can factor
$\phi$ through $M_2(\FF_l) \rtimes \im(\tilde{\rho_1})$. For doing
this, 
\[
\rho_1(\sigma) = (1 + l^r \mu(\sigma)) \rho_2(\sigma)
\]
for some $\mu(\sigma) \in M_2(\Z_l)$, for all $\sigma \in
\Gal(\overline \Q/K)$. Define the map $\varphi:\Gal(\bar{\QQ}/K)
\mapsto M_2(\FF_l) \rtimes \im(\tilde{\rho_1})$ by
\[
\varphi(\sigma)= (\mu(\sigma), \tilde \rho_1(\sigma)) \pmod l.
\]
Then $\phi(\sigma) \equiv \trace(\mu(\sigma) \tilde \rho_1(\sigma))
\pmod l$, i.e. $\phi(A,C)=\trace(AC)$ on $M_2(\FF_l) \rtimes
\im(\tilde{\rho_1})$. An easy computation shows that the group
structure on the semidirect product corresponds to the action by
conjugation, i.e.
\[
\mu(\sigma \tau) \equiv \mu(\sigma) + \tilde \rho_1(\sigma)^{-1}  \mu(\tau) \tilde
\rho_1(\sigma) \pmod l.
\]
Let $\tilde \mu$ denote the composition of $\mu$ with reduction modulo
$l$. The condition $\det(\rho_1) = \det(\rho_2)$ implies that
$\im(\tilde \mu) \subset M_2^0(\FF_l) := \{M \in M_2(\FF_l) \, : \,
\trace(M)\equiv 0 \pmod l\}$, hence it has order at most $l^3$.

Assume that $l=2$ and $\im(\tilde \rho_i) = S_3$, then 

\[
M_2^0(\FF_2)\rtimes S_3 \simeq S_4 \times C_2.
\]

This can be proved in different ways, we give an explicit isomorphism
for latter considerations. Take the isomorphism between $M_2(\FF_2)$
and $S_3$ given by
\begin{eqnarray*}
(12) & \mapsto & \left(\begin{smallmatrix} 0& 1\\ 1& 0 \end{smallmatrix}
     \right), \\
(13) & \mapsto & \left(\begin{smallmatrix} 1& 0\\ 1& 1 \end{smallmatrix}
     \right). \\ 
\end{eqnarray*}
Take $\set{ \left(\begin{smallmatrix}1&1\\ 0 & 1
\end{smallmatrix}\right), \left(\begin{smallmatrix}1&0\\ 1 & 1
\end{smallmatrix}\right), \left(\begin{smallmatrix}1&0\\ 0 & 1
\end{smallmatrix}\right)}$ as a basis for $M_2^0(\FF_2)$.
It is clear that the action of $S_3$ on the last element is trivial.
If we denote $v_1,v_2$ the first two elements of the basis and $v_3$
their sum, the action of $\sigma \in S_3$ on the Klein group
$C_2 \times C_2$ (spanned by $v_1$ and $v_2$) is $\sigma(v_i) =
v_{\sigma(i)}$. Since $S_4 \simeq S_3 \ltimes (C_2 \times C_2)$ with
the same action as described above we get the desired isomorphism.

Clearly the elements of $S_3 \times \left( \begin{smallmatrix} 0 & 0\\
0 & 0 \end{smallmatrix} \right)$, $\set{1} \times M_2^0(\FF_2)$ and
$\set{\sigma \in S_3 \st \sigma^2 =1} \times \left(
\begin{smallmatrix}1 &0 \\ 0 & 1
\end{smallmatrix} \right)$ go to $0$ by $\phi$. It can be seen that
all the other elements have non-trivial image (which correspond to
the elements of order $4$ or $6$ on $S_4 \times C_2$). So we need to
compute all possible extensions of $L$ with Galois group over $K$
isomorphic to $S_4 \times C_2$ and primes $\id{p} \subset K$ with
inertial degree $4$ or $6$ on each field.

\begin{remark} In the proof given above one starts with a $\mod \ell^r$
congruence between the traces of $\rho_1$ and $\rho_2$ and uses the
fact that this implies that the two $\mod \ell^r$ representations are
isomorphic. This result is proved in \cite{serre3} (Theorem $1$) but only
with the assumption that the residual $\mod \ell$ representations are
absolutely irreducible. In fact, it is false in the residually
reducible case, and this is one of the reasons why the above method
does not extend to the case of residual image cyclic of order $3$.
When the residual representations are reducible there are
counter-examples to this claim even assuming that they are
semi-simple. We thank Professor J.-P. Serre for pointing out the
following counter-example to us: take $\ell=2$ and consider two
characters $\chi$ and $\chi'$ defined $\mod 2^r$ such that they agree
$\mod 2^{r-1}$ but not $\mod 2^r$. Then $\chi \oplus \chi$ and $\chi'
\oplus \chi'$ are two-dimensional Galois representations defined $\mod
2^r$ having the same trace but they are not isomorphic.
\end{remark}

\subsection{Second case: the image is a $2$-group}

This case was treated on \cite{livne}, where the author proves the
next Theorem:

\begin{thm}
\label{Livne}
 Let $K$ be an imaginary quadratic field, $S$ a finite set of primes
 of $K$ and $E$ a finite extension of $\Q_2$. Denote by $K_S$ the
 compositum of all quadratic extensions of $K$ unramified outside $S$
 and by $\bfP_2$ the maximal prime ideal of $\Om_E$. Suppose
 $\rho_1,\rho_2:\Gal(\overline \QQ/K)\to\GL_2(E)$ are continuous
 representations, unramified outside $S$, satisfying:
\begin{enumerate}[1.]
\item $\trace(\rho_1)\equiv\trace(\rho_2)\equiv 0 \pmod {\bfP_2}$ and
  $\det(\rho_1)\equiv\det(\rho_2) \pmod {\bfP_2}$.
\item There exists a set $T$ of primes of $K$, disjoint from $S$, for which
\begin{enumerate}[(i)]
\item The image of the set $\{\Frob_t\}$ in the $(\Z/2\Z$-vector
space$)$ $\Gal(K_S/K)$ is non-cubic.
\item $\trace(\rho_1(\Frob_t))=\trace(\rho_2(\Frob_t))$ and
$\det(\rho_1(\Frob_t))=\det(\rho_2(\Frob_t))$ for all $t\in T$.
\end{enumerate}
\end{enumerate}
Then $\rho_1$ and $\rho_2$ have isomorphic semi-simplifications.
\begin{proof} See \cite{livne}.
\end{proof}
\end{thm}

The following result is useful for identifying non-cubic subsets of
  $(\Z/2\Z)$-vector spaces.

\begin{prop}
Let $V$ be a vector space over $\Z/2\Z$. Then a function
$f:V\to\Z/2\Z$ is represented by a homogeneous polynomial of degree
$3$ if and only if $\sum_{I\subset\{0,1,2,3\}}f(\sum_{i\in I}v_i)=0$
for every subset $\{v_0,v_1,v_2,v_3\}\subset V$.
\begin{proof} See \cite{livne}.
\end{proof}
\end{prop}
\end{remark}

\subsection{Third case: the image is cyclic of order $3$}
This is a mix of the previous two cases. Let $E$ be a finite extension
of $\QQ_2$ such that its residue field is isomorphic to $\FF_2$.
Suppose $\rho_1,\rho_2:\Gal(\bar \QQ/K) \rightarrow \GL_2(E)$ are
continuous representations such that the residual representations are
isomorphic and have image a cyclic group of order $3$. Let $K_\rho$ be
the fixed field of the residual representations kernel. If we restric
the two representations to $\Gal(\bar \QQ/ K_\rho)$, we get:

\begin{equation*}
\rho_1,\rho_2: \Gal(\bar \QQ/K_\rho) \rightarrow \GL_2(E),
\end{equation*}
whose residual representation have trivial image. Hence we are in
$2$-group case for the field $K_\rho$ and Livne's Theorem \ref{Livne}
applies.

\section{Proof of the Algorithm}

Before giving a proof for each case we make some general
considerations. The image of $\tilde \rho_\E$ is isomorphic to the
Galois group $\Gal(L_\E/K)$. If $\E(K)$ has a two torsion point, its
image is a $2$-group. If not, assume (via a change of variables) that
the elliptic curve has equation
\[
\E:y^2=x^3+a_2x^2+a_4x+a_6
\]
and denote by $\alpha,\beta,\gamma$ the roots of
$x^3+a_2x^2+a_4x+a_6$. Using elementary Galois theory it can be seen
that $L_\E=K[\alpha-\beta]$. Furthermore, using elementary symmetric
functions, it can be seen that $\alpha-\beta$ is a root of the
polynomial
\[
x^6+x^4(6a_4-2a_2^2)+x^2(a_2^4-6a_2^2a_4+9a_4^2)+4a_6a_2^3-18a_6a_4a_2+4a_4^3-a_4^2a_2^2+27a_6^2.
\]
If this polynomial is irreducible over $K$, the image of
$\tilde{\rho}_\E$ is isomorphic to $S_3$ while if it is reducible, the
image is isomorphic to $C_3$.

Note that under the isomorphism between $S_3$ and $\GL_2(\FF_2)$ given
on the previous section, the order $1$ or $2$ elements of $S_3$ have
even trace while the order $3$ ones have odd trace.

In the case where the image is not a $2$-group, we need to prove that
the image lies (after conjugation) on an extension $E$ of $\Q_2$ with
residual field $\FF_2$. 

\begin{thm} 
If $\E$ has no Complex Multiplication, then we can chose split primes
of $K$, $\bfP_i$, $i=1,2$ such that if
$\alpha_{\bfP_i},\beta_{\bfP_i}$ denote the roots of the
characteristic polynomial of $\Frob_{\bfP_i}$, then the field
  $E=\Q[\alpha_{\bfP_i}]$ has inertial degree $1$ at $2$ and
  $\alpha_{\bar{\bfP_i}} + \beta_{\bar{\bfP_i}} \neq 0$. In
  particular, if $\trace(\rho_\E(\Frob_{\bfP_i})) =
  \trace(\rho_{\pi,2}(\bfP_i))$ then by Taylor's argument (see Remark \ref{Taylor}),
  $\im(\tilde{\rho}_{\pi,2}) \subset \GL_2(\FF_2)$.
\end{thm}
\begin{proof} Since $\E$ has no Complex Multiplication, if $F$ is any
  quadratic field extension of $\Q_2$, the set of primes $\bfP$ such
  that $\Q_2[\alpha_{\bfP}] = F$ has positive density (see for example
  Exercise $(3)$ of \cite{Serre2}). Also, the set of primes $\bfP$
  such that $\alpha_{\bfP}+\beta_{\bfP}=0$ has density zero (since
  $\E$ has no complex multiplication, see \cite{Serre4}), so we can
  find primes $\bfP$ such that $\Q_2[\alpha_{\bfP}] =F$ and
  $\alpha_{\bar{\bfP}} + \beta_{\bar{\bfP}} \neq 0$.  The fields
  $F_1$ and $F_2$ obtained adding the roots of the polynomials
  $x^2+14$ and $x^2+6$ to $\Q_2$ are two ramified extensions of
  $\Q_2$. Their composition is a degree $4$ field extension (since the
  prime $2$ is totally ramified on the composition of these extensions
  over $\Q$). Since the set of primes innert on $K$ have density zero,
  we can chose prime ideals $\bfP_1$ and $\bfP_2$ whose extensions of
  $\Q_2$ are isomorphic to $F_1$ and $F_2$. 
\end{proof}
Actually we search for the first two primes such that $2$ has no
inertial degree on the extension obtained adding to $\Q$ the roots of
their Frobenius automorphisms.

\medskip

The first step of the algorithm is to prove that the residual
representations are indeed isomorphic so as to apply Faltings-Serre's
method. In doing this we need to compute all extensions of a fixed
degree ($2$ or $3$ in our case) with prescribed ramification. Since we
deal with abelian extensions, we can use class field theory.
\begin{thm}
If $L/K$ is an abelian extension unramified outside the set of places
$\set{\id{p}_i}_{i=1}^n$ then there exists a modulus $\id{m} =
\prod_{i=1}^n \id{p}_i^{e(\id{p}_i)}$ such that $\Gal(L/K)$ corresponds to
  a subgroup of the ray class group $Cl(\Om_K,\id{m})$.
\end{thm}

Since we are interested in the case $K$ an imaginary quadratic field,
all the ramified places of $L/K$ are finite ones, hence $\id{m}$ is an
ideal on $\Om_K$. A bound for $e(\id{p})$ is given by the following result.

\begin{prop}
Let $L/K$ be an abelian extension of prime degree $p$. If $\id{p}$
ramifies on $L/K$, then

\[
\left \{\begin{array}{cl}
e(\id{p}) = 1 & \text{if } \id{p} \nmid p\\
2
\le e(\id{p}) \le \left \lfloor\frac{p e(\id{p}|p)}{p-1} \right \rfloor +1 & \text{if }\id{p} \mid
 p.
\end{array} \right.
\]
\label{cohen}
\end{prop}

\begin{proof} See \cite{cohen} Proposition $3.3.21$ and Proposition $3.3.22$.
\end{proof}




To distinguish representations, given a character $\psi$ of a
ray class field we need to find a prime ideal $\id{p}$ with
$\psi(\id{p}) \neq 1$. Let $\psi$ be a character of $Cl(\Om_K,\id{m}_K)$
of prime order $p$. Take any branch of the logarithm over $\CC$ and
identify $\log(\{\xi_p^i\})$ with $\ZZ/p\ZZ$ in any way (where $\xi_p$
denotes a primitive $p$-th root of unity).
\begin{prop} Let $K$ be a number field, $\id{m}_K$ a modulus and
$Cl(\Om_K,\id{m}_K)$ the ray class field for $\id{m}_K$. Let 
$\set{\psi_i}_{i=1}^n$ be a basis of order $p$ characters of
$Cl(\Om_K,\id{m})$ and $\set{\id{p}_j}_{j=1}^{n'}$ be prime ideals of
$\Om_K$ such that
\[
\<\log(\psi_1(\id{p}_j)), \ldots,\log(\psi_n(\id{p}_j))>_{j=1}^{n'} = (\ZZ/p\ZZ)^n.
\]
Then for every non trivial character $\psi$ of $Cl(\Om_K,\id{m})$ of
order $p$, $\psi(\id{p}_j) \neq 1$ for some $1 \le j \le n$.
\label{basischoice}
\end{prop}
\begin{proof}
Suppose that $\psi(\id{p}_j)=1$ for $1\le j \le m$. Since
$\set{\psi_i}_{i=1}^n$ is a basis, there exists exponents
$\varepsilon_i$ such that
\[
\psi = \prod_{i=1}^n \psi_i^{\varepsilon_i}
\]
Taking logarithm and evaluating at $\id{p}_j$ we see that
$(\varepsilon_1, \ldots \varepsilon_n)$ is a solution of the
homogeneous system
\[
\left(\begin{array}{ccc}
\log(\psi_1(\id{p}_1)) & \ldots & \log(\psi_n(\id{p}_1))\\
\vdots & &\vdots\\
\log(\psi_1(\id{p}_m)) & \ldots & \log(\psi_n(\id{p}_m))\end{array}
\right ).
\]
Since
$\set{(\log(\psi_1(\id{p}_j)),\ldots,\log(\psi_n(\id{p}_j)))}_{j=1}^m$
span $(\ZZ/p\ZZ)^n$, the matrix has maximal rank, hence
$\varepsilon_i=0$ and $\psi$ is the trivial character.
\end{proof}
\begin{remark}
A set of prime ideals satisfying the conditions of the previous
Proposition always exists by Tchebotarev's density theorem.
\end{remark}

\subsection{Residual image isomorphic to $S_3$} 

\begin{remark} If the residual representation is absolutely
irreducible, we can apply a descent result (see Corollaire 5 in
\cite{serre3}, which can be applied because the Brauer group of a
finite field is trivial) and conclude that since the traces are all in
$\FF_2$ the representation can be defined (up to isomorphism) as a
representation with values on a two-dimensional $\FF_2$-vector
space. Thus, the image can be assumed to be contained in
$\GL_2(\FF_2)$ and because of the absolute irreducibility assumption
we conclude that the image has to be isomorphic to $S_3$.
\end{remark}

Furthermore, we have the following result,

\begin{thm} If the image is absolutely ireducible, then the field $E$
  can be taken to be $\Q_2$.
\end{thm}

\begin{proof} This follows from the same argument as the previous
  Remark. See also Corollary of \cite{FLT}, page $256$. 
\end{proof}

\begin{remark}
Once we prove that the residual representation of $\rho_{\pi,2}$ has
image greater than $C_3$ we automatically know that it can be defined
on $\GL_2(\Z_2)$.
\end{remark}

We have the $2$-adic Galois representations $\rho_\E$ and $\rho_f$ and
we want to prove that they are isomorphic. We start by proving that
the reduced representations are isomorphic. For doing this we compute
all quadratic extensions of $K$ using Class Field theory and
Proposition \ref{cohen}. Let $K_\E$ denote the quadratic extension of
$K$ contained on $L_\E$. Following the ideas of step $(5)$ of the
previous case, we can prove that $L_f$ (the fixed field of the kernel
of $\rho_f$) contains no quadratic extension of $K$ or contains $K_\E$
(note that an ideal with inertial degree $3$ on $L_\E$ splits on
$K_\E$). This is done on steps $(1)-(4)$.
\begin{remark}
Let $P(x)$ denote the degree $3$ polynomial in $K[x]$ whose roots are
the $x$-coordinates of the points of order $2$ of $\E$. The fact that
the splitting field of $P(x)$ is an $S_3$ extension allows us to
compute how primes decompose on $K_\E$ knowing how they decompose on
the cubic extension $K_P$ of $K$ obtained by adjoining any root of
$P(x)$. The factorization as well as the values of $\psi(\id{p})$ are
given by the next table:
\[
\begin{tabular}{|c|c|c|r|}
\hline
$\id{p} \Om_{K_P}$ & $\id{p} \Om_{K_\E}$ & $\id{p} \Om_{L_\E}$ & $\psi(\id{p})$\\
\hline
$\id{p}_1 \id{p}_2 \id{p}_3$ & $\id{q}_1 \id{q}_2$ &$\id{t}_1\ldots \id{t}_6$ & $1$\\
$\id{p}_1 \id{p}_2$ & $\id{p}$ & $\id{t}_1 \id{t}_2 \id{t}_3$ & $-1$\\
$\id{p}$ & $\id{q}_1 \id{q}_2$ & $\id{t}_1 \id{t}_2$ &$1$\\
\hline
\end{tabular}
\]
\begin{proof} The last two cases are clear (since the inertial degree
is multiplicative and it is at most $3$). The not so trivial case is
the first one. Since $L_\E/K$ is Galois, $\id{p} \Om_{L_\E}$ has $3$ or
$6$ prime factors. Assume 
\begin{equation}
\id{p} \Om_{L_\E} = \id{q}_1 \id{q}_2 \id{q}_3. 
\label{cannot happen}
\end{equation}
Then it must be the case that (after relabeling the ideals
if needed) if $\sigma$ denotes one order three element in
$\Gal(L_\E/K)$, $\sigma(\id{q}_1) = \id{q}_2$ and $\sigma^2(\id{q}_1) =
\id{q}_3$. Since the decomposition groups $D(\id{q}_i|\id{p})$ have
order $2$ and are conjugates of each other by powers of $\sigma$, they
are disjoint and they are all the order $2$ subgroups of $S_3$. Since
$K_P$ is a degree $2$ subextension of $L_\E$, it is the fixed field of
an order $2$ subgroup. Without loss of generality, assume $K_P$ is the
fixed field of $D(\id{q}_1|\id{p})$. If
we intersect equation $(\ref{cannot happen})$ with $\Om_{K_P}$ we get
\[
\id{p} \Om_{K_P} = (\id{q}_1 \cap \Om_{K_P})(\id{q}_2 \cap \Om_{K_P})(\id{q}_3 \cap \Om_{K_P}).
\]
We are assuming that $(\id{q}_i \cap \Om_{K_P}) \neq (\id{q}_j \cap
\Om_{K_P})$ if $i \neq j$. Let $\tau$ be the non trivial element on
$D(\id{q}_1|\id{p})$, so $\tau$ acts trivially on $K_P$. In particular,
$\tau$ fixes $\id{q}_2 \cap \Om_{K_P}$ and $\tau(\Om_{L_\E}) =
\Om_{L_\E}$ then $\tau(\id{q_2}) = \tau((\id{q}_2 \cap \Om_{K_P})
\Om_{L_\E}) = \id{q_2}$ which contradicts that $D(\id{q}_1|\id{p}) \cap
D(\id{q}_2|\id{p}) = \{1\}$.
\end{proof}
\label{rem}
\end{remark}
Next we need to discard the $C_3$ case. Let $\id{m}_K$ be as described
on step $(1)$ of the algorithm, and $Cl(\Om_K,\id{m}_K)$ be the ray
class field. Suppose that $\tilde{\rho}_f$ has image isomorphic to
$C_3$. Let $\chi$ be (one of) the cubic character of
$Cl(\Om_K,\id{m}_K)$ corresponding to $L_f$. Let
$\{\chi_i\}_{i=1}^m$ be a basis of cubic characters of
$Cl(\Om_K,\id{m}_K)$. We look for prime ideals $\{\id{p}_j\}_{j=1}^{m'}$
that are inert on $K_\E$ or split completely on $L_\E$ (that is, they have
order $1$ or $2$ on $S_3$ and in particular have even trace for the
residual representation $\tilde{\rho}_\E$) and such that
$\<(\log(\chi_1(\id{p}_j)),\ldots,\log(\chi_m(\id{p}_j)))>_{j=1}^{m'}
=(\ZZ/3\ZZ)^m$. There exists such ideals by Tchebotarev's density
Theorem. By Proposition \ref{basischoice}, there exists an index $i_0$
such that $\chi(\id{p}_{i_0}) \neq 1$, hence
$\trace(\tilde{\rho}_f(\id{p}_{i_0})) \equiv 1 \pmod 2$ while
$\trace(\tilde{\rho}_\E(\id{p}_{i_0})) \equiv 0 \pmod 2$. Step $(6)$
discards this case.

Once we know that $\tilde{\rho}_f$ has $S_3$ image with the same
quadratic subfield as $\tilde{\rho}_\E$, we take $K_\E$ as the base
field and proceed in the same way as the previous case. This is done
on steps $(7)$ and $(8)$. 

At this point we already decided whether the two residual
representations are isomorphic or not. If they are, we can apply
Faltings-Serre's method explained on the previous section. We look for
quadratic extensions $\tilde L$ of $L$ unramified outside $\id{m}_L$ 
such that its normal closure is isomorphic to $S_4$ or $S_3 \times
C_2$. 

Applying Faltings-Serre's method, we need to compute all fields
$\tilde L$ with Galois group $\Gal(\tilde L/K) \simeq S_4 \times C_2$
in the $S_3$ case and $\Gal(\tilde L/K) \simeq A_4 \times C_2$ in the
$C_3$ case. The group $S_4 \times C_2$ fits in the exact sequences

\[
1 \rightarrow C_2 \times C_2 \rightarrow S_4 \times C_2 \rightarrow
S_3 \times C_2 \rightarrow 1
\]
and
\[
1 \rightarrow C_2  \rightarrow S_4 \times C_2 \rightarrow S_4 \rightarrow 1.
\]
The group $A_4 \times C_2$ fits in the exact sequence
\[
1 \rightarrow C_2 \times C_2 \rightarrow A_4 \times C_2 \rightarrow
C_3 \times C_2 \rightarrow 1.
\]
Every element of order $4$ or $6$ on $S_4 \times C_2$ maps to an
element of order $4$ on $S_4$ or to an element of order $6$ on $S_3
\times C_2$ under the previous surjections, while any element of order
$6$ on $A_4 \times C_2$ maps to an element of order $6$ on $C_3 \times
C_2$.  We can restrict ourselves to compute normal extensions of $L$ with
Galois group $S_4$ or $S_3 \times C_2$ in the $S_3$ case and normal
extensions of $L$ with Galois group $C_3 \times C_2$ on the $C_3$
case. Note that in all cases the extensions are obtained by computing
the normal closure of a quadratic extension. 

Let $\id{m}_L$ be a modulus on $L$ invariant under the action of
$\Gal(L/K)$. Then $\Gal(L/K)$ has an action on $Cl(\Om_L,\id{m}_L)$ and it
induces an action on the set of characters of the group. Concretely,
if $\psi$ is a character on $Cl(\Om_L,\id{m}_L)$ and $\sigma \in
\Gal(L/K)$, $\sigma . \psi = \psi \circ \sigma$.
\begin{lemma} If $\psi$ is a character on $Cl(\Om_L,\id{m}_L)$ that
  corresponds to the quadratic extension $L[\sqrt{\alpha}]$ and
  $\sigma \in \Gal(L/K)$ then $\sigma^{-1}.\psi$ corresponds to
  $L[\sqrt{\sigma(\alpha})]$. 
\end{lemma}

\begin{proof} The character is characterized by its value on non-ramified
  primes. Let $\id{p}$ be a non-ramified prime on
  $L[\sqrt{\alpha}]/L$. It splits on $L[\sqrt{\alpha}]$ if and
  only if $\psi(\id{p}) = 1$. If $\id{p}$ does not divide the
  fractional ideal $\alpha$, this is equivalent to $\alpha$ being a square modulo
  $\id{p}$. But for $\sigma \in \Gal(L/K)$, $\alpha$ is a square
  modulo $\id{p}$ if and only if $\sigma(\alpha)$ is a square modulo
  $\sigma(\id{p})$ hence the extension $L[\sqrt{\sigma(\alpha)}]$
  corresponds to the character $\sigma^{-1} . \psi$.
\end{proof}

\begin{prop} Let $L/K$ be a Galois extension with $\Gal(L/K) \simeq
  S_3$ and $\psi$ a quadratic character of $Cl(\Om_L,\id{m}_L)$ with
  $\id{m}_L$ as above.
\label{Serre-s3}
\begin{enumerate}
\item The quadratic extension of $L$ corresponding to $\psi$ is Galois
if and only if $\sigma .\psi = \psi$ for all $\sigma \in \Gal(L/K)$.
\item The quadratic extension of $L$ corresponding to $\psi$ has
  normal closure isomorphic to $S_4$ if and only if the elements
  fixing $\psi$ form an order $2$ subgroup and $(\psi)(\sigma.\psi) =
  \sigma^2. \psi$, where $\sigma$ is any order $3$ element in $\Gal(L/K)$.
\end{enumerate}
\end{prop}

\begin{proof} Let $L[\sqrt{\alpha}]$ be a quadratic extension of
$L$. The normal closure (with respect to $K$) is the field 
\[
\tilde L = \prod _{\sigma \in \Gal(L/K)} L[\sqrt{\sigma(\alpha)}]
\]
(where by the product we mean the smallest field containing all of
them inside $\bar \QQ$). In particular $\Gal(\tilde
L/L)$ is an abelian $2$-group. By the previous proposition, if
$L[\sqrt{\alpha}]$ corresponds to the quadratic character $\psi$ then
the other ones correspond to the characters $\sigma . \psi$ where
$\sigma \in \Gal(L/K)$.

The first assertion is clear. To prove the second one, the condition
$(\psi)(\sigma.\psi) = \sigma^2 \psi$ and $\psi$ being fixed by an order
$2$ subgroup implies that $[\tilde L : L] = 4$. Hence the group
$\Gal(\tilde{L}/K)$ fits in the exact sequence
\[
1 \rightarrow C_2 \times C_2 \rightarrow \Gal(\tilde{L}/K) \rightarrow
S_3 \rightarrow 1.
\]
In particular $\Gal(\tilde{L}/K)$ is isomorphic to the semidirect
product $S_3 \ltimes (C_2 \times C_2)$, with the action given by a
morphism $\Theta: S_3 \rightarrow GL_2(\FF_2)$. Its kernel is a normal
subgroup, hence it can be $GL_2(\FF_2)$ (i.e. the trivial action),
$\<\sigma>$ (the order $3$ subgroup) or trivial. The condition on the
stabilizer of $\psi$ forces the image of $\Theta$ to contain an order
$3$ element, hence the kernel is trivial. Up to inner automorphisms,
there is a unique isomorphism from $GL_2(\FF_2)$ to itself (and
morphisms that differ by an inner automorphism give isomorphic groups)
hence $\Gal(\tilde{L}/K) \simeq S_4$ as claimed.
\end{proof}

\begin{remark}
On the $S_4$ case of the last proposition, the condition on the action
of $\sigma$ is necessary. Consider the extension $L =
\Q[\xi_3,\sqrt[3]{2}]$ where $\xi_3$ is a primitive third root of
unity. It is a Galois degree $6$ extension of $\Q$ with Galois group
$S_3$. Take as generators for the Galois group the elements $\sigma,
\tau$ given by
\begin{eqnarray*}
\sigma: \xi_3 \mapsto  \xi_3 & \text{ and } & \sigma: \sqrt[3]{2}
\mapsto \xi_3 \sqrt[3]{2} \\
\tau: \xi_3 \mapsto  \xi_3^2 & \text{ and } & \tau: \sqrt[3]{2}
\mapsto \sqrt[3]{2} \\
\end{eqnarray*}
The extension $L\left [\sqrt{1+\sqrt[3]{2}}\right ]$ is clearly fixed by $\sigma$,
but its normal closure has degree $8$ over $L$ since
$\sqrt{1+\xi_3^2\sqrt[3]{2}}$ is not on the field
$L\left [\sqrt{1+\sqrt[3]{2}},\sqrt{1+\xi_3\sqrt[3]{2}}\right ]$ as
can be easily checked.
\end{remark}

To compute all such extensions, we use Class Field Theory and
Proposition \ref{Serre-s3}. To compute the $S_3 \times C_2$ extensions
we follow the method of the $C_3$ case. The only difference is that we
check for invariance under an order $3$ element plus invariance under
an order $2$ element. This is done on steps $(11)-(15)$.

At last, we need to compute the quadratic extensions whose normal
closure has Galois group isomorphic to $S_4$. Using the second part of
Proposition \ref{Serre-s3}, we need to compute quadratic characters
$\chi$ such that $\chi (\sigma.\chi) (\sigma^2.\chi) =1$ (where
$\sigma$ denotes an order three element of $\Gal(L/K)$) and also whose
fixed subgroup under the action of $\Gal(L/K)$ has order $2$. Let $S$ denote the set
of all such characters. Since $\sigma$ does not act trivially on
elements of $S$, we find that $\chi$, $\sigma.\chi$ and
$\sigma^2.\chi$ are three different elements of $S$ that give the same
normal closure. Then we can write $S$ as a disjoint union of three
sets. Furthermore, since $\sigma$ acts transitively (by multiplication
on the right) on the set of order $2$ elements of $S_3$, we see that
\[
S = V_{\tau} \cup V_{\tau \sigma} \cup V_{\tau \sigma^2}
\]
where $V_{\tau}$ denotes the quadratic characters of $S$ invariant
under the action of $\tau$ and the union is disjoint. Hence each one
of these sets is in bijection with all extensions $\tilde L$ of
$L$. We compute one subspace and use Proposition \ref{basischoice} on
this subspace, noting that the elements of order $4$ correspond to primes
that are inert on any of the three extensions of $L$ (corresponding
to $\chi$, $\sigma.\chi$ and $\sigma^2.\chi$) hence we consider not one
prime above $\id{p} \subset \Om_K$ but all of them. This is done on
steps $(16)-(19)$.

\subsection{Trivial residual image or residual image isomorphic to $C_2$}

The first step is to decide if we can take $E$ to be an extension of
$\QQ_2$ with residue field $\FF_2$ so as to apply Livne's Theorem
\ref{Livne}. Once this is checked, the algorithm is divided into two
parts. Let $\rho_\E,\rho_f:\Gal(\bar \Q/K)\to\GL_2(\Z_2)$ be given,
with the residual image of $\rho_\E$ being either trivial or
isomorphic to $C_2$. Steps $(2)$ to $(6)$ serve to the purpose of
seeing whether $\rho_f$ has also trivial or $C_2$ residual image or
not. Note that the output of step $(6)$ does not say that the residual
representations are actually the same, but they have isomorphic
semisimplifications (in this case it is equivalent to say that the
traces are even). For example, there can be isogenous curves, one of
which has trivial residual image and the other has $C_2$ residual
image.

Suppose we computed the ideals of steps $(2)-(5)$ and $\tilde \rho_f$ has
even trace on the Frobenius of these elements. We claim that $\rho_f$
has residual image either trivial or $C_2$. Suppose on the contrary
that $\rho_f$ has residual image isomorphic to $C_3$. Let $L_2/K$ be
the cyclic extension of $K$ corresponding (by Galois theory) to the
kernel of $\tilde \rho_f$. This corresponds to a cubic character
$\chi$ of $Cl(\Om_K,\id{m}_K)$. An easy calculation shows that if
$\id{p}\subset\Om_K$ is a prime ideal not dividing $\id{m}$, then
$\chi(\id{p})=1$ if and only if
$\trace(\tilde\rho_f(\Frob_\id{p}))=0$. This implies that
$\chi(\id{p}_j)=1$ for each $j=1,\dots,m$. But $\chi$ is a non-trivial
character, then by Proposition 4.6 we get a contradiction.

Similarly, suppose that the residual image of $\rho_2$ is $S_3$. Let
$L_2/K$ be the $S_3$ extension of $K$ corresponding (by Galois theory)
to the kernel of $\tilde\rho_f$, and $M_2/K$ its unique quadratic
subextension. The extension $L_2/M_2$ corresponds to a cubic character
$\chi$ of $Cl(\Om_{M_2},\id{m}_{M_2})$ and the proof follows the
previous case.

Steps $(7)-(10)$ check if the representations are indeed isomorphic
once we know that the traces are even using Theorem \ref{Livne}. We
need to find a finite set of primes $T$, which will only depend on
$K$, and check that the representations agree at those primes. In
the algorithm and in the theorem, we identify the group $\Gal(K_S/K)$
with the group of quadratic characters of $Cl(\Om_K,\id{m})$.  In step
$(7)$, we compute the image of $\Frob_\id{p}\in \Gal(K_S/K)$ via this
isomorphism and compute enough prime ideals so as to get a
non-cubic set of $\Gal(K_S/K)$. Then the representations are
isomorphic if and only if the traces at those primes agree.

\subsection{Residual image isomorphic to $C_3$}
Let $K$ be an imaginary quadratic field and let 
\[
\rho_\E,\rho_f: \Gal(\bar{\Q}/K) \rightarrow GL_2(\Q_2)
\]
be the Galois representations attached to $\E$ and $f$ respectively.

The first step is to decide if we can take $E$ to be an extension of
$\QQ_2$ with residue field $\FF_2$. Once this is checked, we need to
prove that the residual representation $\tilde{\rho}_f$ has image
isomorphic to $C_3$. For doing this we start proving that it has no
order $2$ elements on its image. If such an element exists, there
exists a degree $2$ extension of $K$ unramified outside $2
\id{n}(\E)\id{n}(f) \overline{\id{n}(f)} \Disc(K)$.  We use
Proposition \ref{cohen} and Class Field Theory to compute all such
extensions. Once a basis of the quadratic characters is chosen, we
apply Proposition \ref{basischoice} to find a set of ideals such that
for any quadratic extension, (at least) one prime $\id{q}$ on the set
is inert on it. Since the residual image is isomorphic to a subgroup
of $S_3$, $\tilde{\rho}_f(\id{q})$ has order exactly $2$. In
particular its trace is even. If $\trace(\tilde{\rho}_f(\id{p}))$ is
odd at all primes, $\im(\tilde{\rho}_f)$ contains no order $2$
elements. Also since $\trace(\identity) \equiv 0 \pmod{2}$ we see that
$\tilde{\rho}_f$ cannot have trivial image hence its image is
isomorphic to $C_3$. This is done on steps $(2)$ to $(5)$ of the
algorithm.

To prove that $\tilde{\rho}_f$ factors through the same field as
$\tilde{\rho}_\E$ we compute all cubic Galois extensions of $K$. This
can be done using Class Field Theory again, and this explains the
choice of the modulus on step $(1)$, so as to be used for both
quadratic and cubic extensions. Note that the characters $\chi$ and
$\chi^2$ give raise to the same field extension. If $\psi_\E$ denotes
(one of) the cubic character corresponding to $L_\E$, we extend it to a
basis $\{\psi_\E,\chi_i\}_{i=1}^m$ of the cubic characters of
$Cl(\Om_K,\id{m}_K)$ and find a set of prime ideals
$\{\id{p}_j\}_{j=1}^{m'}$ such that
$\<(\log(\chi_1(\id{p}_j)),\ldots,\log(\chi_n(\id{p}_j)))>_{j=1}^{m'}
=(\ZZ/3\ZZ)^m$ and $\psi_\E(\id{p}_j)=1$ for all $1 \le j \le m'$. 

If $\chi$ is a cubic character corresponding to $L_f$,
$\chi=\psi_\E^\varepsilon \varkappa$, where $\varkappa =
\prod_{i=1}^n\chi_i^ {\varepsilon_i}$. If $\varkappa = 1$, then $\chi
= \psi_\E$ or $\psi_\E^ 2$ and we are done. If not, by Proposition
\ref{basischoice}, there exists an index $i_0$ such that
$\varkappa(\id{p}_{i_0}) \neq 1$. Furthermore, since
$\psi_\E(\id{p}_{i_0})=1$, $\chi(\id{p}_{i_0}) \neq 1$. Hence
$\trace(\tilde{\rho}_f(\id{p}_{i_0})) \equiv 1 \pmod 2$ and
$\trace(\tilde{\rho}_\E(\id{p}_{i_0})) \equiv 0 \pmod 2$.

At this point we already decided whether the two residual
representations are isomorphic or not. If they are, we can apply
Livne's Theorem \ref{Livne} to the field $L_E$ which is the last step
of the algorithm.


\section{Examples}

In this section we present three examples of elliptic curves over
imaginary quadratic fields, one for each class of residual image and
show how the method works. The first publications comparing elliptic
curves with modular forms over imaginary quadratic fields are due to
Cremona and Whitley (see \cite{cremona-whitley}), where they consider imaginary
quadratic fields with class number $1$. The study was continued by
other students of Cremona. We consider some examples of Lingham's
Ph.D. thesis where examples are computed for quadratic fields with
class number $3$ to show that our computations work on general
situations. For all examples, we take $K = \QQ[\sqrt{-23}]$ and we
denote $\omega = \frac{1+\sqrt{-23}}{2}$.

All computations were done using the PARI/GP system (\cite{PARI2}).
On the next section we present the commands used to check our examples
so as to serve as a guide for further cases. The routines written by
the authors can be downloaded from \cite{CNT}.

\subsection{Image isomorphic to $S_3$.} Let $\E$ be the elliptic
curve with equation
\[
\E:y^2+\omega x y + y = x^3+(1-\omega)x^2-x
\]
It has conductor $\id{n}_\E = \bar{\id{p}}_2 \id{p}_{13}$ where
$\bar{\id{p}}_2 = \<2,1-\omega>$ and $\id{p}_{13} =
\<13,8+\omega>$. There is an automorphic form of level
$\id{n}_f=\bar{\id{p}}_2 \id{p}_{13}$ (denoted by $f_2$ on
\cite{lingham} table $7.1$) which is the candidate to correspond to
$\E$. We know that $f$ has a $2$-adic Galois representation attached
whose $L$-series local factors agree at all primes except (at most)
$\set{\id{p}_{23},\bar{\id{p}}_2,\id{p}_2,\id{p}_{13},\bar{\id{p}}_{13}}$. Let
$\rho_\E$ be the $2$-adic Galois representation attached to $\E$. Its
residual representation has image isomorphic to $S_3$ as can easily be
checked by computing the extension $L_\E$ of $K$ obtained adding the
coordinates of the $2$-torsion points.
Using the routine {\tt setofprimes} we find that the set of primes of
$\QQ[\sqrt{-23}]$ above $\set{3, 5, 7, 29, 31, 41, 47}$ is enough for
proving that the residual representations are isomorphic and that the
$2$-adic representations are isomorphic as well. Note that the natural
answer would be the set $\set{3, 5, 7, 11, 19, 29, 31, 37}$, but since
some of these ideals have norm greater than $50$, they are not on
table $7.1$ of \cite{lingham}. This justifies our first list.

To prove that the answer is correct, we apply the algorithm
described on section \ref{s3-case}:
\begin{enumerate}
\item Since $2$ is unramified on $K/\QQ$, the modulus is $\id{m}_K =
2^3 13\sqrt{-23}$.
\item The ray class group is isomorphic to $C_{396} \times C_{12}
  \times C_2 \times C_2 \times C_2 \times C_2$. Using Remark \ref{rem}
  we find that the character $\psi$ on the computed basis corresponds
  to $\chi_3$, where $\{\chi_i\}$ is the dual basis of quadratic characters.
\item The extended basis is
  $\{\psi,\chi_1,\chi_2,\chi_4,\chi_5,\chi_6\}$. Computing some prime
  ideals, we find that the set $\{\bar{\id{p}}_3,
  \id{p}_5,\bar{\id{p}}_{29},\id{p}_{31},\id{p}_{47}\}$
  has the desired properties (using Remark \ref{rem} we know that the
  primes with inertial degree $3$ are the ones on the third
  case). On distinguishing one ideal from its conjugate we follow the
  notation of \cite{lingham} for consistency (although it may not be
  the order of GP's output).
\item Table $7.1$ of \cite{lingham} shows that $\trace(\tilde{\rho}_f(\Frob_{\id{p}}))$ is
  odd on all such primes $\id{p}$.
\item The group of cubic characters has as dual basis for
  $Cl(\Om_K,\id{m}_K)$ the characters $\set{\chi_1,\chi_2}$,
  i.e. $\chi_i(v_j) = \delta_{i,j} \xi_3$, where $\xi_3$ is a
  primitive cubic root of unity and $\delta_{i,j}$ is Dirac's delta
  function. The ideals $\id{p}_3$ and $\id{p}_7$ are inert on the
  quadratic subextension of $L_\E$ and
\[
\langle(\log(\chi_1(\id{p}_3)),\log(\chi_2(\id{p}_3))),
  (\log(\chi_1(\id{p}_7)),\log(\chi_2(\id{p}_7)))\rangle =
  (\ZZ/3\ZZ)^2
\]
\item From table $7.1$ (of \cite{lingham}) we see that
  $\trace(\rho_f(\Frob_{\id{p}_3}))$ is even, hence $\tilde{\rho}_f$ has image $S_3$
  with the same quadratic subfield as $\tilde{\rho}_\E$.
\item The field $K_\E$ can be given by the equation $x^4 + 264*x^3 +
26896*x^2 + 1244416*x + 21958656$. The prime number $2$ is ramified on
$K_\E$, and factors as $2 \Om_{K_\E}= \id{p}_{2,1}^2 \id{p}_{2,2}$. The
prime number $13$ is also ramified and factors as $13 \Om_{K_\E} =
\id{p}_{13,1}^2 \id{p}_{13,2} \id{p}_{13,3}$. The prime number $23$ is
ramified, but has a unique ideal dividing it on $K_\E$. The modulus to
consider is $\id{m}_{K_\E} = \id{p}_{2,1}^5 \id{p}_{2,2}^3\id{p}_{13,1}
\id{p}_{13,2} \id{p}_{13,3} \id{p}_{23}$.
\item $Cl(\Om_{K_\E},\id{m}_{K_\E}) \simeq C_{792} \times C_{12} \times
  C_{12} \times C_{12} \times C_4 \times C_2 \times C_2 \times
  C_2$. We claim that $\psi_\E=\chi_1^2\chi_4$, where $\chi_i$ is the
  dual basis for cubic characters of $Cl(\Om_{K_\E},\id{m}_{K_\E})$. To
  prove this, we use Remark \ref{cannot happen}. We know that
  $\id{p}_3$ is inert on $K_\E$ and $\psi_\E(\id{p}_3)=1$. The prime
  number $7$ is inert on $K_\E$ hence $\psi_\E(\id{p}_7)=1$; the prime
  $37$ is inert in $K$, but splits as a product of two ideals on
  $K_\E$. Then $\psi_\E(\id{p}_{37})=1$ on both ideals. There is a unique
  (up to squares) character vanishing on them, and this is $\psi_\E$.

The basis $\{\psi_\E, \chi_1,\chi_2,\chi_3\}$ extends $\{\psi_\E\}$ to a basis of
  cubic characters. The point here is that the characters
  $\chi_i$ need not give Galois extensions over $K$. A character gives
  a Galois extension if and only if its modulo is invariant under the
  action of $\Gal(K_\E/K)$. The characters $\chi_1,\chi_3,\chi_4$ do
  satisfy this property, hence the subgroup of cubic characters of
  $Cl(\Om_{K_\E},\id{m}_{K_\E})$ with invariant conductor has rank $3$. A
  basis is given by $\{\psi_\E,\chi_1,\chi_3\}$. If we evaluate
  $\chi_1$ and $\chi_3$ at the prime above $\id{p}_3$ and $\id{p}_7$
  we see that they span the $\ZZ/3\ZZ$-module. We already
  compared the residual traces on these ideals, hence the two residual
  representations are indeed isomorphic.
\item We compute an equation for $L_\E$ over $\QQ$. From the ideal
  factorizations $2 \Om_{L_\E} = \id{q}_{2,1}^2
  \id{q}_{2,2}^2\id{q}_{2,3}^2\id{q}_{2,4}^3$, $13 \Om_{L_\E} =
  \id{q}_{13,1}^2
  \id{q}_{13,2}^2\id{q}_{13,3}^2\id{q}_{13,4}\id{q}_{13,5}$ and $23
  \Om_{L_\E} = \id{q}_{23,1}^2 \id{q}_{23,2}^2 \id{q}_{23,3}^2$ we take 
\begin{center}
$\id{m}_{L_\E} = \id{q}_{2,1}^5 \id{q}_{2,2}^5 \id{q}_{2,3}^5\id{q}_{2,4}^7 
\id{q}_{13,1} \id{q}_{13,2}\id{q}_{13,3}\id{q}_{13,4}\id{q}_{13,5}\id{q}_{23,1} \id{q}_{23,2} \id{q}_{23,3}$
 \end{center}
as the modulus and compute the ray class group
$Cl(\Om_{L_\E},m_{L_\E})$. It has $18$ generators (see the {\it GP
  Code} section).
\item We compute the Galois group $\Gal(L_\E/K)$, and chose an order $3$
  and an order $2$ elements from it.
\item We compute the kernels of the system and find out that
  the kernel for the order $3$ element has dimension $8$.
\item The kernel for the order $2$ element has dimension $11$. 
\item The intersection of the previous two subspaces has dimension
  $6$. It is generated by the characters
\begin{center}
  $\{\chi_1, \chi_2\chi_5, \chi_2\chi_3\chi_6\chi_7,
  \chi_3\chi_4\chi_9, \chi_{12}\chi_{13}\chi_{14},
  \chi_8\chi_{10}\chi_{12}\chi_{15}\chi_{17}\}$.
\end{center}
\item The ideals above $\{3,5,11,29,31\}$ satisfy that their
logarithms span the $\ZZ/2\ZZ$ vector space.
\item If we look at table $7.1$ of \cite{lingham}, we found that the ideal above
  $11$ is missing since it has norm $121$, but we can replace it by
  the ideals above $47$. So we checked that the two representations
  agree on order $6$ elements.
\item The space of elements satisfying the condition on the order $3$
  element has dimension $10$. 
\item The intersection of the two subspaces has dimension $5$. A basis
  is given by the characters
\begin{center}
  $\{\chi_1\chi_2\chi_4,\chi_1\chi_2\chi_6,\chi_3\chi_{10}\chi_{11}\chi_{14},\chi_3\chi_{16},
  \chi_1\chi_{10}\chi_{11}\chi_{12}\chi_{13}\chi_{17}\}$. 
\end{center}
\item The prime ideals above $\set{3,7,19,29,31}$ do satisfy the
  condition, but since the prime $19$ is inert on $K$, its norm is
  bigger than $50$. Nevertheless, we can replace it by the primes
  above $41$ which are on Table $7.1$ of \cite{lingham}.
\item Looking at table $7.1$ of \cite{lingham} we find that the two
  representations are indeed isomorphic.
\item From the same table we see that the factors at the primes
  $\bar{\id{p}}_2$, $\id{p}_2$, $\id{p}_{23}$, $\bar{\id{p}}_{13}$ and $\id{p}_{13}$ also agree hence the
  two $L$-series are the same.
\end{enumerate}

If the stronger version of Theorem \ref{Harris-Taylor} saying that the
level of the Galois representaion equals the level of the automorphic
form is true, the set of primes to consider can be diminished removing
the primes above $37$ on the second set of primes.

\subsection{Trivial residual image or image isomorphic to $C_2$}

Let $\E$ be the elliptic curve over $K$ with equation
\[ \E: y^2+\omega xy=x^3-x^2-(\omega+6)x.\]
According to ~\cite{lingham}, the conductor of $\E$ is
$\id{p}_2 \overline{\id{p}}_3$, where $\id{p}_2=\langle 2,\omega\rangle$ and
$\id{p}_3=\langle 3,-1+\omega\rangle$.

Using the routine {\tt setofprimes}, we find that the set 
\[\{5,7,11,13,17,19,29,31,41,47,59,61,67,71,83,89,97,101,127,131,139,151,163,\]
\[179,197,211,233,239,277,311,349,353,397,439,443,739,1061,1481\}\]
is enough for checking modularity. We will confirm that this is indeed
the case.  

\begin{enumerate}
\item The primes above $29$ and $41$ prove that the residual
  representation of the automorphic form lies on $\GL_2(\FF_2)$,
  because the values of $a_{\bfP_{29}},a_{\overline{\bfP_{29}}},
  a_{\bfP_{41}}$ and $a_{\overline{\bfP_{41}}}$ are $6,6,-2,2$
  respectively. Then degree $4$ extension of $\QQ_2$ has equation $x^4
  + 8x^3 + 144x^2 + 512x + 896$, and it is totally ramified.
\item The modulus is $\id{m}_K=2^3 3^2\sqrt{-23}$, and the ray class
group $Cl(\Om_K,\id{m}_K)\simeq C_{198}\times C_6\times C_2\times
C_2\times C_2\times C_2$.

\item $-\; (4)$ There are 64 quadratic (including the trivial)
extensions of $K$ with conductor dividing $\id{m}_K$. We calculate
each one, with the corresponding ray class group described in the
algorithm; we pick a basis of cubic characters of each group, and
evaluate them at each prime in $\{5,7,11,13,17,19,29,31\}$.  We find
that this set is indeed enough for proving whether $\tilde{\rho}_f$ has
residual image trivial or isomorphic to $C_2$. \setcounter{enumi}{4}

\item Since $\trace(\rho_f(\Frob_{\id{p}})) \equiv 0 \pmod 2$ for
  the primes on the previous set (see \cite{lingham} table $7.1$) we get that
  the residual image is trivial or isomorphic to $C_2$.

\item$-\; (7)$ The set
\[\{5,7,13,29,31,41,47,59,61,67,71,83,89,97,101,127,131,139,151,163,\]
\[179,197,211,233,239,277,311,349,353,397,439,443,739,1061,1481\}\]
is enough. In order to see this, we must check that the Frobenius at
all the primes of $K$ above these ones cover $\Gal(K_S/K)\setminus
\{\identity\}$.  We calculate a basis
$\{\psi_1,\psi_2,\psi_3,\psi_4,\psi_5\}$ for the quadratic characters
of $Cl(\Om_K,\id{m}_K)$, and compute, for $\id{p}$ a prime of $K$ above one
of these primes, $(\log\psi_1(\id{p}),\dots,\log\psi_5(\id{p}))$. We simply
check that this set of coordinates has 63 elements, so the primes we
listed are enough. \setcounter{enumi}{7}

\item The primes listed are not on Table $7.1$ of \cite{lingham}. We
 asked Professor John Cremona to compute the missing values so as
  to prove modularity for this curve.
\end{enumerate}

If the stronger version of Theorem \ref{Harris-Taylor} saying that the
level of the Galois representaion equals the level of the automorphic
form is true, the set of primes to consider can be diminished to the
primes above $\{5, 7, 11, 13, 17, 23, 29, 31, 47, 59, 71, 101, 131\}$.

\subsection{Image isomorphic to $C_3$}

Let $\E$ be the elliptic curve with equation
\[
\E: y^2 = x^3 - \omega x^2 + (4 \omega -1)x-5.
\]
This curve is taken from Table $7.3$ of \cite{lingham}. It has
conductor $\id{p}_2^2 \bar{\id{p}}_2^3$, where $\id{p}_2 = \<2,\omega>$ and
$\bar{\id{p}}_2 = \<2,1+\omega>$.
There is an automorphic form $f$ (denoted $f_4$ on \cite{lingham})
which has level $\id{n} = \id{p}_2^2 \bar{\id{p}}_2^3$ and is the
candidate to correspond to $\E$. From Section $2$ we know that $f$ has
a Galois $2$-adic representation $\rho_f$ attached to it, whose
$L$-series agree at all primes with the possible exceptions
$\id{p}_2, \bar{\id{p}}_2$ and $\id{p}_{23}$, where $\id{p}_{23}$ is the
unique ideal or norm $23$ on $K$. Let $\rho_\E$ denote the $2$-adic
Galois representation attached to $\E$. Its residual representation has
image isomorphic to $C_3$ as can be easily proved by computing the
extension $L_\E$ of $K$ obtained adjoining the $2$-torsion points
coordinates.

Using the GP routine {\tt setofprimes} (which can be downloaded from
\cite{CNT}), we find that the set of primes of $\QQ[\sqrt{-23}]$ above
\[\{3, 5, 7, 11, 13, 17, 19, 29, 31, 41, 47, 53, 59, 73, 79, 83,
  89, 101, 131, 167, 173, 191, 211,\]
\[ 223, 241, 271, 307, 317, 347, 421,
  463, 593, 599, 607, 617, 691, 809, 821, 853, 877, 883,\]
\[ 997, 1151,
  1481, 1553, 1613, 1789, 1871, 2027, 2339, 2347, 2423, 2693, 3571,
  4831\}\] is enough for proving that the residual representations
are isomorphic and that the $2$-adic representations are isomorphic as
well.

To prove that the result is correct, we apply the algorithm described on 
Section \ref{c3-case}
\begin{enumerate}
\item The primes above $59$ and $173$ are enough to prove that the
  residual representation of the automorphic form lies on
  $\GL_2(\FF_2)$. The values of $a_{\bfP_{59}},a_{\overline{\bfP_{59}}},$
  $a_{\bfP_{173}}$ and $a_{\overline{\bfP_{173}}}$ are $-12,4,-6,10$
  respectively. The Frobenius polinomials of the first primes split in
  $\Q_2$, hence the Galois representations lies on a quadratic
  extension of $\Q_2$. And $2$ ramifies for the second primes, as claimed.
\item Since $2$ is unramified on $K/\QQ$, the modulus is $\id{m}_K=
  2^3.\sqrt{-23}$. We compute this ray class group and find that
  $Cl(\Om_K,\id{m}_K) \simeq C_{66} \times C_2 \times C_2 \times C_2$.
\item There is a unique (up to squares) order three character,
  hence a unique cubic extension of $K$ unramified outside $\id{m}_K$ so
  it corresponds to $L_\E$.
\item Let $\set{\chi_1,\ldots,\chi_4}$ be a set of generators of the
  order two characters of $Cl(\Om_K,\id{m}_K)$ with respect to the
  previous isomorphism. By computing their values at prime
  ideals of $\Om_K$ we found that the set
  $C=\set{\id{p}_3,\bar{\id{p}}_3,\id{p}_{13},\bar{\id{p}}_{13}}$ satisfies
  the desired properties.
\item The traces of the Frobenius at these primes are odd (see
\cite{lingham} table $7.1$).
\item Since there are no other order three characters, we have that
  $\rho_\E \simeq \rho_f$.
\item As in the previous example, Livne's method implies that the
  primes above the primes in the set\\
$\{3, 5, 7, 11, 13, 17, 19, 29, 31, 41, 47, 53, 59, 73, 79, 83, 89,
101, 131, 167, 173, $\\
$191,211, 223, 241, 271, 307, 317, 347, 421, 463,
593, 599, 607, 617, 691, 809, $\\
$821, 853, 877, 883, 997, 1151, 1481,
1553, 1613, 1789, 1871, 2027, 2339, 2347, $\\
$2423, 2693, 3571, 4831 \}$ 

are enough to prove modularity.
\item Most of the primes listed are not on Table $7.1$ of
  \cite{lingham}. We asked Professor John Cremona to compute the
  missing values so as to prove modularity for this curve.
\end{enumerate}
\begin{remark} This case is rather special, since the extension $L$ is
  Galois over $\Q$. In particular the residual representation $\tilde \rho_\E$
  and $\tilde \rho_{\bar{\E}}$ (where bar denotes conjugation) are
  isomorphic. This allows working with extensions over $\QQ$ and
  avoid working with relative Galois extensions.
\end{remark}

\section{GP Code}

In this section we show how to compute the previous examples with our
routines and the outputs.

\subsection{Image $S_3$}

\begin{verbatim}
? read(routines);
? K=bnfinit(w^2-w+6);
? Setofprimes(K,[w,1-w,1,-1,0],[2,13])
Case = S_3
Class group of K:  [396, 12, 2, 2, 2, 2]
Primes for discarding other quadratic extensions: [3, 5, 11, 29, 31]
Primes discarding C_3 case:  [3, 7]
The ray class group for K_E is  [792, 12, 12, 12, 4, 2, 2, 2]
Cubic character on K_E basis:  [0; 0; 0; 1]
Primes proving C_3 extension of K_E:  [3, 7, 37]
Class group of L:  [2376, 12, 12, 12, 4, 4, 4, 4, 4, 2, 2, 2, 2, 2, 
2, 2, 2, 2]
%3 = [3, 5, 7, 11, 19, 29, 31, 37]
\end{verbatim}

\subsection{Image isomorphic to $C_2$ or trivial}
\begin{verbatim}
? read(routines);
? K=bnfinit(w^2-w+6);
Case = C_2 or trivial
Primes for proving that the residual representation
	lies on F_2:  [29, 41]
Class group of K:  [198, 6, 2, 2, 2, 2]
There are 64 subgroups of Cl_K of index <= 2
Primes proving C_2 or trivial case [5, 7, 11, 13, 17, 19, 29, 31]
Livne's method output:[5, 7, 11, 13, 17, 19, 29, 31, 37, 41, 47, 
59, 71, 97, 101, 127, 131, 139, 151, 163, 179, 197, 211, 233, 239, 
277, 311, 349, 353, 397, 439, 443, 739, 1061, 1481]
%3 = [5, 7, 11, 13, 17, 19, 29, 31, 37, 41, 47, 59, 71, 97, 101, 
127, 131, 139, 151, 163, 179, 197, 211, 233, 239, 277, 311, 349, 
353, 397, 439, 443, 739, 1061, 1481]
\end{verbatim}

\subsection{Trivial residual image or image isomorphic to $C_3$}

\begin{verbatim}
? read(routines);
? K=bnfinit(w^2-w+6);
? Setofprimes(K,[0,-w,0,4*w-1,-5],[2])
Case = C_3
Primes for proving that the residual representation
	lies on F_2:  [59, 173]
Class group of K:  [66, 2, 2, 2]
Primes proving C_3 image:  [59, 173, 3, 13]
Cubic character on K basis:  [;]
Primes proving C_3 extension of K_E:  []
Livne's method output:[3, 5, 7, 11, 13, 17, 19, 29, 31, 41, 47, 
53, 59, 73, 79, 83, 89, 101, 131, 167, 173, 191, 211, 223, 241, 
271, 307, 317, 347, 421, 463, 593, 599, 607, 617, 691, 809, 821, 
853, 877, 883, 997, 1151, 1481, 1553, 1613, 1789, 1871, 2027, 
2339, 2347, 2423, 2693, 3571, 4831]
%3 = [3, 5, 7, 11, 13, 17, 19, 29, 31, 41, 47, 53, 59, 73, 79, 83, 
89, 101, 131, 167, 173, 191, 211, 223, 241, 271, 307, 317, 347, 
421, 463, 593, 599, 607, 617, 691, 809, 821, 853, 877, 883, 997, 
1151, 1481, 1553, 1613, 1789, 1871, 2027, 2339, 2347, 2423, 2693, 
3571, 4831]
\end{verbatim}

\bibliographystyle{alpha}
\bibliography{al}

\begin{thebibliography}{PAR08}

\bibitem[BH]{berger-harcos}
Tobias Berger and Gergely Harcos.
\newblock $\ell$-adic representations associated to modular forms over
  imaginary quadratic fields.
\newblock {\em Int. Math. Res. Not. to appear}.
\newblock arXiv:0707.1338.

\bibitem[CNT]{CNT}
Computational number theory.
\newblock {\tt http://www.ma.utexas.edu/users/villegas/cnt/}.

\bibitem[Coh00]{cohen}
Henri Cohen.
\newblock {\em Advanced topics in computational number theory}, volume 193 of
  {\em Graduate Texts in Mathematics}.
\newblock Springer-Verlag, New York, 2000.

\bibitem[CSS97]{FLT}
Gary Cornell, Joseph~H. Silverman, and Glenn Stevens, editors.
\newblock {\em Modular forms and {F}ermat's last theorem}.
\newblock Springer-Verlag, New York, 1997.
\newblock Papers from the Instructional Conference on Number Theory and
  Arithmetic Geometry held at Boston University, Boston, MA, August 9--18,
  1995.

\bibitem[CW94]{cremona-whitley}
J.~E. Cremona and E.~Whitley.
\newblock Periods of cusp forms and elliptic curves over imaginary quadratic
  fields.
\newblock {\em Math. Comp.}, 62(205):407--429, 1994.

\bibitem[HST93]{harris-taylor}
Michael Harris, David Soudry, and Richard Taylor.
\newblock {$l$}-adic representations associated to modular forms over imaginary
  quadratic fields. {I}. {L}ifting to {${\rm GSp}\sb 4({\bf Q})$}.
\newblock {\em Invent. Math.}, 112(2):377--411, 1993.

\bibitem[Lin05]{lingham}
Mark Lingham.
\newblock {\em Modular forms and elliptic curves over imaginary quadratic
  fields}.
\newblock PhD thesis, University of Nottingham, October 2005.

\bibitem[Liv87]{livne}
Ron Livn{\'e}.
\newblock Cubic exponential sums and {G}alois representations.
\newblock In {\em Current trends in arithmetical algebraic geometry (Arcata,
  Calif., 1985)}, volume~67 of {\em Contemp. Math.}, pages 247--261. Amer.
  Math. Soc., Providence, RI, 1987.

\bibitem[PAR08]{PARI2}
The PARI~Group, Bordeaux.
\newblock {\em PARI/GP, version {\tt 2.4.3}}, 2008.
\newblock available from {\tt http://pari.math.u-bordeaux.fr/}.

\bibitem[Sch06]{schutt}
Matthias Sch{\"u}tt.
\newblock On the modularity of three {C}alabi-{Y}au threefolds with bad
  reduction at 11.
\newblock {\em Canad. Math. Bull.}, 49(2):296--312, 2006.

\bibitem[Ser66]{Serre4}
Jean-Pierre Serre.
\newblock Groupes de {L}ie {$l$}-adiques attach\'es aux courbes elliptiques.
\newblock In {\em Les {T}endances {G}\'eom. en {A}lg\'ebre et {T}h\'eorie des
  {N}ombres}, pages 239--256. \'Editions du Centre National de la Recherche
  Scientifique, Paris, 1966.

\bibitem[Ser68]{Serre2}
Jean-Pierre Serre.
\newblock {\em Abelian {$l$}-adic representations and elliptic curves}.
\newblock McGill University lecture notes written with the collaboration of
  Willem Kuyk and John Labute. W. A. Benjamin, Inc., New York-Amsterdam, 1968.

\bibitem[Ser85]{serre}
Jean-Pierre Serre.
\newblock R\'esum\'e des cours de 1984-1985.
\newblock {\em Annuaire du Coll\`ege de France}, pages 85--90, 1985.

\bibitem[Ser95]{serre3}
Jean-Pierre Serre.
\newblock Representaions lineaires sur des anneaux locaux, d'apres carayol.
\newblock {\em Publ. Inst. Math. Jussieu}, 49, 1995.

\bibitem[Tay94]{taylorII}
Richard Taylor.
\newblock {$l$}-adic representations associated to modular forms over imaginary
  quadratic fields. {II}.
\newblock {\em Invent. Math.}, 116(1-3):619--643, 1994.

\end{thebibliography}
\end{document}